\documentclass[12pt]{amsart}
\usepackage[utf8]{inputenc}
\usepackage{amsmath,amssymb,mathrsfs, amsfonts, amsthm}
\usepackage{comment}
\usepackage{subfig}
\usepackage{graphicx}
\usepackage{enumitem}
\usepackage{yfonts}
\usepackage[margin=1in]{geometry}
\usepackage{color}
\usepackage{url}
\usepackage{framed}
\usepackage{enumitem}
\usepackage{multirow}
\usepackage{setspace}
\usepackage{tikz}
\usetikzlibrary{decorations.pathreplacing}
\allowdisplaybreaks[3]
\newcommand{\N}{\mathbb{N}}
\newcommand{\R}{\mathbb{R}}
\newcommand{\C}{\mathbb{C}}
\newcommand{\Z}{\mathbb{Z}}
\newcommand{\Q}{\mathbb{Q}}

\newcommand{\cG}{\mathcal{G}}

\newtheorem{thm}{Theorem}
\newtheorem{lemma}{Lemma}
\newtheorem{prop}{Proposition}

\newtheorem{rem}{Remark}
\newtheorem{deft}{Definition}

\usepackage{microtype}
\usepackage[hidelinks]{hyperref}
\hypersetup{pdftitle={HRT manuscript: annotated revision},
  pdfauthor={Shuang Guan and Kasso A. Okoudjou}}
\definecolor{ReviewRed}{HTML}{A51C30}
\definecolor{ReviewBlue}{HTML}{1455A0}
\definecolor{ReviewPurple}{HTML}{67328E}

\newtheorem{reviewthm}{Theorem}

\newtheorem{reviewcor}[reviewthm]{Corollary}
\title{The HRT Conjecture for Symmetric Configurations and Real-Valued Functions}
\date{\today}
\author{Shuang Guan}
\address{Department of Mathematics, Tufts University, Medford MA 02155, USA}
\email{shuang.guan@tufts.edu}
\author{Kasso A.~Okoudjou}
\address{Department of Mathematics, Tufts University, Medford MA 02155, USA}
\email{kasso.okoudjou@tufts.edu}

\begin{document}

\begin{abstract}
The Heil--Ramanathan--Topiwala conjecture asserted that every finite
collection of distinct time-frequency shifts of a nonzero
square-integrable function is linearly independent. Recent
counterexamples with Schwartz windows leave open the classification
of configurations and windows for which independence holds.

We prove linear independence for every nonzero window in $L^2(\R)$
and every configuration consisting of finitely many commensurably
spaced points on one line together with two distinct points on a
parallel line. As a consequence, we obtain independence for every
nonzero real-valued window and every four-point configuration
$\{(0,0),(0,1),(s,0),(a,b)\}$ with $abs\ne0$.
The configurations underlying the known four-point counterexamples
can be put in this form by translation and dilation, so they cannot
yield dependence with nonzero real-valued windows.
The proof combines factorization of Laurent polynomials with
product estimates along orbits of irrational rotations.
Finally, for a fixed nonzero window and a fixed three-point system,
we show that the set of time-frequency parameters producing a shift
in its span is compact, has Lebesgue measure zero, has Hausdorff
dimension at most one, and is totally disconnected.

\end{abstract}

\subjclass[2020]{Primary 42C15 Secondary 42C40, 37A30}
\keywords{HRT conjecture, time-frequency analysis}

\maketitle

\section{Introduction}\label{sec:intro}
For $a,b\in\R$, let
$T_ag(t)=g(t-a), M_bg(t)=e^{2\pi ibt}g(t).
$
Given a nonzero $g\in L^2(\R)$ and a finite set of distinct points
$\Lambda\subset\R^2$, the associated finite Gabor system is
\[
\cG(g,\Lambda)=\{M_bT_ag:(a,b)\in\Lambda\}.
\]
The Heil--Ramanathan--Topiwala conjecture asserted that every such system
is linearly independent \cite{heil1996linear}. In August 2026, Faulhuber,
Petersen, van Velthoven, and Voigtlaender disproved the conjecture by
constructing twelve linearly dependent time-frequency shifts of a nonzero
Schwartz function \cite{faulhuber2026linear}. Shortly afterward, Oussa
\cite{oussa2026fourpoint} and, independently, Dai, Deng, Shi, Wu, and Yang
\cite{dai2026minimum} constructed four-point counterexamples, also with
Schwartz windows. Since systems with at most three distinct shifts are
linearly independent, four is the smallest possible cardinality of a
counterexample. Jasper and Mixon subsequently constructed counterexamples
with exponential or faster decay \cite{jasper2026exponential}.

These developments leave open the classification of four-point
configurations that guarantee linear independence for every nonzero
$L^2$ window. In this direction, Oussa has developed a program
that organizes the noncollinear four-point problem into nine
geometric and arithmetic cases, relative to a chosen noncollinear
triple
\cite{oussa2026criticalmixedarithmeticfourpointhrt,
oussa2026minimalhrtcounterexamplesharp,
oussa2026fourpoint,
oussa2026leancertifiedfourpointhrtresults}.
Six cases have been settled positively, while the counterexamples
of Oussa and Dai et al.\ belong to the same seventh case.
The remaining two cases are open in full generality.
Thus, although the universal four-point conjecture is false,
determining precisely which configurations admit counterexamples
remains an open problem.

The remaining classification problem concerns configurations that force
independence for every nonzero $L^2$ window, and window classes that
exclude dependence. Classical positive results include independence for collinear
configurations \cite{heil1996linear}, for finite subsets of translates of
lattices \cite{linnell1999neumann}, and for $(2,2)$ configurations
\cite{demeter2010linear,demeter2012proof}. Here an $(m,n)$ configuration
consists of $m+n$ distinct points on two distinct parallel lines, with
$m$ points on one line and $n$ on the other. Liu established independence
for a full-measure class of $(1,3)$ configurations and every nonzero
$L^2$ window \cite{liu2019proof}. For background and further positive
results under assumptions on the window or configuration, we refer to
\cite{Groc2001,heil2006linear,heil2015hrt}.

We give three results in this direction. The first
establishes independence for an arbitrary finite set of commensurably
spaced points on one line together with two arbitrary distinct points
on a parallel line, see Theorem~\ref{review:two-lines}. The second establishes linear independence for real-valued windows for a class
of four-point configurations that includes the recent counterexamples
after translation and dilation, see Corollary~\ref{review:real}. The final main result of our paper, describes the size and topology
of the set of dependent extensions of a fixed three-point system, see Theorem~\ref{review:extensions}. 

\begin{samepage}
\begin{reviewthm}[Two parallel lines with commensurable spacings]
\label{review:two-lines}
Let $F\subset\Z$ be finite and nonempty, let $a\ne0$, and let
$b_1,b_2\in\R$ be distinct. Then, for every nonzero $g\in L^2(\R)$,
the system $\cG(g,\Lambda)$ is linearly independent, where
\[
\Lambda=\{(0,k):k\in F\}\cup\{(a,b_1),(a,b_2)\}.
\]
\end{reviewthm}
\end{samepage}

We say that finitely many points on a line have commensurable spacings
if they lie in a set $z_0+h\Z u$, where $h>0$ and $u\ne0$ is a direction
vector for the line. Translations of the configuration and symplectic
changes of coordinates reduce any configuration described above to the
form in Theorem~\ref{review:two-lines}; the latter changes are implemented
by unitary metaplectic operators on the window. Thus the theorem applies
to two parallel lines in any direction.

The symmetric configurations
\[
\Lambda_n=\{(0,k):-n\le k\le n\}\cup\{(a,b),(a,-b)\},
\qquad ab\ne0,
\]
are a useful normalized family. Indeed, the shear
\[
S_z(x,y)=(x,y+zx),\qquad z=-\frac{b_1+b_2}{2a},
\]
fixes the first line pointwise and sends the second pair to
$\{(a,b),(a,-b)\}$ with $b=(b_1-b_2)/2\ne0$. Symmetry of that pair
therefore imposes no additional restriction. When $n=1$, the theorem
completes the symmetric $(3,2)$ case previously treated under arithmetic
and, in one case, real-valuedness hypotheses in
\cite[Theorem~6]{okoudjou2019extension}.

Commensurability makes the trigonometric polynomial in a hypothetical
dependence relation a Laurent polynomial in a single variable.
After normalizing the translation step, the relation has the form
\[
g(t)P(t)=g(t-1)Q(t),
\]
where $P$ has commensurable frequencies and $Q$ has two terms.
Factoring $P$ reduces its orbit products to products associated with
linear factors. Roots on the unit circle are treated with the product
estimates of Demeter and Zaharescu \cite{demeter2012proof}; roots off
the circle are treated by uniform Riemann-sum estimates. Combining the
finitely many exceptional sets yields simultaneous estimates for the
forward and backward products along a common subsequence. Periodicity
or a conjugate-orbit argument then contradicts the decay of an $L^2$
window along almost every integer orbit.

The symmetric $(3,2)$ result also yields the following application of
the restriction principle of \cite{okoudjou2019extension}.

\begin{reviewcor}[Real-valued four-point systems]\label{review:real}
Let $0\ne g\in L^2(\R)$ be real-valued. If
\[
\widetilde\Lambda=\{(0,0),(0,1),(s,0),(a,b)\}
\]
consists of four distinct points with $abs\ne0$, then
$\cG(g,\widetilde\Lambda)$ is linearly independent.
\end{reviewcor}

Conjugating a putative dependence and eliminating the purely translated
term gives a nontrivial
dependence on the symmetric five-point configuration
$\{(0,0),(0,\pm1),(a,\pm b)\}$, contradicting
Theorem~\ref{review:two-lines}. This argument does not settle the
four-point problem for arbitrary configurations and real-valued
windows: a general metaplectic normalization need not preserve
real-valuedness.

Both counterexample constructions in \cite{oussa2026fourpoint,dai2026minimum} use a set
of the form
\[
\Lambda_{\alpha,\beta}
=\{(0,0),(\alpha,\beta/2),(\alpha+1,\beta/2),
       (\alpha,(\beta+1)/2)\}.
\]
Translating by $-(\alpha,\beta/2)$ and applying the area-preserving
dilation $D(x,y)=(x/2,2y)$ gives
\[
D\bigl(\Lambda_{\alpha,\beta}-(\alpha,\beta/2)\bigr)
=\{(0,0),(0,1),(1/2,0),(-\alpha/2,-\beta)\}.
\]
Translation of a configuration preserves linear independence with the
same window, and the unitary dilation associated with $D$ preserves
real-valuedness. Corollary~\ref{review:real} therefore shows that these
specific configurations cannot yield dependence with a nonzero
real-valued $L^2$ window, although they do yield dependence with the
complex-valued Schwartz windows constructed in those papers.

Our final result concerns the extension principle introduced in \cite{okoudjou2019extension}. Fix
$0\ne g\in L^2(\R)$ and three distinct points $\Lambda_3\subset\R^2$,
and set
\[
V=\operatorname{span}\cG(g,\Lambda_3),\qquad
\mathcal E(g,\Lambda_3)=\{\lambda\in\R^2:\pi(\lambda)g\in V\},
\]
where $\pi(a,b)=M_bT_a$. The space $V$ has dimension three.
The set $\mathcal E(g,\Lambda_3)$ contains $\Lambda_3$, and its
remaining points are exactly the fourth points producing dependence.

\begin{reviewthm}[Structure of dependent extensions]\label{review:extensions}
For every nonzero $g\in L^2(\R)$ and every set $\Lambda_3$ of three
distinct points, $\mathcal E(g,\Lambda_3)$ is compact, has Lebesgue
measure zero, has Hausdorff dimension at most one, and is totally
disconnected. Consequently, for each fixed nonzero $g$, the system
$\cG(g,\Lambda)$ is linearly independent for Lebesgue-almost every
ordered four-point configuration in $(\R^2)^4$.
\end{reviewthm}

This result uses classical independence theorems rather than
Theorem~\ref{review:two-lines}. Four points of $\mathcal E$ yield
dependence because their shifts lie in $V$. Hence no line
contains four points of $\mathcal E$, and no two distinct parallel
lines each contain two of its points. These restrictions give the
measure and dimension bounds through slicing, and rule out nontrivial
connected subsets. The exceptional set in the almost-everywhere
statement is allowed to depend on $g$.

The rest of the paper is structured as follows. Section~\ref{sec:2} develops the product estimates.
Section~\ref{sec:3-2config} proves the symmetric $(3,2)$ result and
the real-valued corollary. Section~\ref{sec:2n1} extends the argument
to commensurable configurations with any finite number of points on
the first line. Section~\ref{sec:E} proves the structural assertions
about $\mathcal E(g,\Lambda_3)$.

\section{Preliminary Estimates and Reduction}\label{sec:2}

In this section, we consider the case $n=1$, that is, the symmetric $(3,2)$-configuration. We assume for the sake of contradiction that for certain functions $0\neq g \in L^2(\R)$ the HRT conjecture fails and derive some key equations. We also recall some results of \cite{demeter2012proof} on products
of exponential differences along orbits of irrational rotations, and prove a key product estimate on which our proof relies.

Let $0\neq g \in L^2(\R)$ and let 
\[\Lambda_0=\{(0,0),(0,1),(0,-1),(a,b),(a,-b)\},\quad 
 ab\neq 0.\]
Reflecting the time variable, if necessary, and relabeling the
symmetric pairs, we may assume that $a,b>0$. Define the unitary
dilation
\[
(D_ag)(t)=a^{1/2}g(at).
\]
Since
\[
D_aM_\beta T_\alpha
   =M_{a\beta}T_{\alpha/a}D_a,
\]
the system $\cG(g,\Lambda_0)$ is linearly independent if and only
if $\cG(D_ag,\Lambda)$ is linearly independent, where
\[
\Lambda=
\begin{pmatrix}a^{-1}&0\\0&a\end{pmatrix}\Lambda_0
=\{(0,0),(0,a),(0,-a),(1,v),(1,-v)\},
\qquad v=ab>0.
\]
We henceforth denote the dilated window again by $g$, reserving
$\Lambda_0$ for the original configuration and $\Lambda$ for
the normalized configuration.

The following relation holds for a.e. $t \in \R$:
\begin{equation}\label{eq:ch3_relation} 
\begin{cases} 
g(t)P(t) = g(t-1)Q(t) \\ 
P(t) = c_1+ c_2 e^{2 \pi i a t}+c_3 e^{-2\pi iat}\\ 
Q(t) = d_1 e^{2 \pi i vt} + d_2 e^{-2\pi i vt} 
\end{cases}
\end{equation}
where there exists $c_1,c_2,c_3,d_1,d_2 \in \C$. Notice that all these coefficients are nonzero because the $(2,2)$ and the $(1,3)$ configurations with equally spaced points on a line have already been resolved \cite{demeter2012proof, heil1996linear, linnell1999neumann}.

 Let $N_0$
be the exceptional null set where \eqref{eq:ch3_relation} fails,
enlarged to contain the zeros of $P$ and $Q$. These zero sets
have measure zero because $P$ and $Q$ are nonzero exponential
polynomials. Set
\[
\Omega=\R\setminus\bigcup_{m\in\Z}(N_0+m).
\]
Then $\Omega$ has full measure and is invariant under integer
translations. At every point of $\Omega$, the recurrence holds
and both $P$ and $Q$ are nonzero. Consequently,
\[
g(t)=0 \quad\Longleftrightarrow\quad g(t-1)=0,
\qquad t\in\Omega.
\]
The set $\{t\in\Omega:g(t)\ne0\}$ is therefore invariant under
integer translations. Since $g\ne0$, it follows that
\[
A=\{t\in\Omega\cap[0,1]:g(t)\ne0\}
\]
has positive measure.

Moreover, Tonelli's theorem gives
\[
\int_0^1\sum_{n\in\Z}|g(t+n)|^2\,dt=\|g\|_2^2<\infty.
\]
After removing a null subset of $A$, we may therefore assume
that
\[
H_N(t):=\sup_{\substack{n\in\Z\\ |n|\ge N}}|g(t+n)|
\longrightarrow0
\qquad (N\to\infty)
\]
for every $t\in A$. By Egorov's theorem, there is a measurable
set $S\subset A$ of positive measure on which this convergence
is uniform. Thus

\begin{equation}\label{eq:L2decay}
\lim_{\substack{|N|\to\infty\\N\in\Z}}
\sup_{t\in S}|g(t+N)|=0.
\end{equation}

Since $S+\Z\subset\Omega$, the recurrence holds at every point
of $S+\Z$, and
\[
g(t+n)\ne0,\qquad P(t+n)\ne0,\qquad Q(t+n)\ne0
\qquad (t\in S,\ n\in\Z).
\]

As we will see, the proof relies on understanding the infinite product of $|P|$
along forward and backward orbits of $t \pm \Z$. Estimation of exactly this type was developed by Demeter and Zaharescu in their work of the $(2,2)$ case \cite{demeter2012proof}, and we now recall two results from their work that are required in our arguments.

Let $\{x\}$ and
$\langle x\rangle:=\operatorname{dist}(x,\Z)$ denote the
fractional part of $x$ and its distance to the nearest integer.
Fix an irrational $\alpha$, and let $p_k/N_k$ be its regular
continued-fraction convergents. Their denominators satisfy
$N_k\le N_{k+1}$, and
\[
p_kN_{k-1}-p_{k-1}N_k=(-1)^{k-1},
\qquad
\left|\alpha-\frac{p_k}{N_k}\right|
\le \frac{1}{N_kN_{k+1}}.
\]
Choose an infinite set $E\subset\N$ and a constant
$D=D(\alpha)>0$ such that
\[
\frac{N_k}{N_{k+1}}
\le D\min_{1\le j\le k}\frac{N_j}{N_{j+1}},
\qquad k\in E.
\]
After discarding finitely many indices, we may assume that
$N_k>100$ for every $k\in E$. Define
\[
\frac{1}{M_k}
:=N_k^2\left|\alpha-\frac{p_k}{N_k}\right|,
\qquad
\mathcal N(\alpha):=\{N_k:k\in E\}.
\]
Throughout the following product estimates, we take $k\in E$
and $0<\delta<1/100$.

\begin{prop}\cite[Proposition 2.1]{demeter2012proof}\label{prop2.1} Define $N:= N_k, p:= p_k, M:= M_k$. Then for each $x \in [0,1]$ such that
    $$\min \bigg\{ \frac{\langle x\rangle}{N}, \langle x - n \alpha\rangle,\langle x -\tfrac{n}{N}\rangle: 1 \leq n \leq N\bigg\} \geq \frac{\delta}{N}$$
    we have 
    $$\prod_{n=1}^N |e^{2\pi i x} - e^{2\pi i \alpha n}| \sim_{\delta} 1.$$
\end{prop}

\begin{rem}   Proposition~2.1 of \cite{demeter2012proof} is stated for odd
indices. To obtain the even-index case, put
$\theta=\{\alpha\}$ and $\beta=1-\theta$.
The continued fractions of $\alpha$ and $\theta$ have the
same denominators. If $\theta=[0;a_1,a_2,\ldots]$, then
\[
\beta=
\begin{cases}
[0;1,a_1-1,a_2,\ldots], & a_1\ge2,\\
[0;a_2+1,a_3,\ldots],   & a_1=1.
\end{cases}
\]
Thus, apart from finitely many initial terms, the same
denominators occur for $\beta$ with the opposite index parity.
The successive denominator ratios also agree after this
index shift. Consequently, the condition defining $E$
remains valid for the corresponding indices of $\beta$,
after increasing the fixed constant $D$ if necessary.

For $y=\{-x\}$, complex conjugation gives
\[
\prod_{n=1}^{N}
\left|e^{2\pi ix}-e^{2\pi i\alpha n}\right|
=
\prod_{n=1}^{N}
\left|e^{2\pi iy}-e^{2\pi i\beta n}\right|.
\]
Furthermore,
\[
\langle y\rangle=\langle x\rangle,
\qquad
\langle y-n\beta\rangle=\langle x-n\alpha\rangle,
\]
and reflection permutes the grid
$\{n/N:1\le n\le N\}$ modulo one. Hence the separation
hypothesis is preserved. Applying the odd-index estimate
to $\beta$ proves the required estimate for even indices,
with constants independent of $N$.

\end{rem}

The following is an immediate corollary.

\begin{prop}\cite[Corollary 2.3]{demeter2012proof}\label{coro2.3}
    Let $A, B \in \C$ with $|A| = |B|=1$. Let also $\alpha$ and $N$ be as in Proposition \ref{prop2.1}. Define 
    $$P(x) = A + Be^{2\pi i \alpha x}$$
    Then for each $0 < \epsilon < 1$ there exist $p_1(\epsilon,A,B,\alpha),p_2(\epsilon,A,B,\alpha)>0$ and a set $\mathcal{P}(A,B,\epsilon,\alpha,N)\subset [0,1]$ with measure at least $1-\epsilon$ such that for each $y \in \mathcal{P}(A,B,\epsilon,\alpha,N)$
    $$p_2(\epsilon,A,B,\alpha) \leq \prod_{n = -N}^{-1}|P(y+n)|\leq p_1(\epsilon,A,B,\alpha)$$
    and 
    $$p_2(\epsilon,A,B,\alpha) \leq \prod_{n = 0}^{N-1}|P(y+n)|\leq p_1(\epsilon,A,B,\alpha)$$
    Notice that the set $\mathcal{P}$ is allowed to depend on N, but the constants $p_1,p_2$ do not.
\end{prop}

Factoring $P$ in the variable $e^{2\pi iat}$ reduces the product
estimates to bounds for two linear factors. Propositions~\ref{prop2.1}
and~\ref{coro2.3} handle roots on the unit circle, after removing a
small exceptional set, while uniform Riemann-sum estimates handle
roots off the circle. Combining these bounds yields the following lemma.

\begin{lemma}\label{Sec3} Suppose $a\not\in \Q$. 
    Let $P(t)$ be as defined in Equation~\eqref{eq:ch3_relation}, and let
    $\mathcal{N}=\mathcal{N}(a)\subset\N$ denote the infinite set of denominators
    $N = N_k$ produced by Proposition~\ref{prop2.1} (i.e.\ $k \in E$
    $N_k>100$). For each $0 < \epsilon<\frac{1}{2}$, there exist constants
    $C_1,C_2>0$ that depend only on $\epsilon,c_1,c_2,c_3,a$ but not on $N$,
    and for each $N \in \mathcal{N}$ a set $\mathcal{P}(N)\subset [0,1]$ with
    measure at least $1-\epsilon$, such that for all $N \in \mathcal{N}$ and
    all $t_0,t_1 \in \mathcal{P}(N)$:
    $$C_2 \leq \frac{\prod_{n=-N}^{-1}|P(t_1+n)|}{\prod_{n=1}^{N}|P(t_0+n)|} \leq C_1.$$
\end{lemma}

\begin{proof}

Let $\omega = e^{2 \pi i a t}$ and denote by $\rho_1, \rho_2$ the roots of the
quadratic $c_2 \omega^2 + c_1 \omega + c_3 = 0$. Then
$$e^{2\pi i at }P(t) = c_1 e^{2 \pi i a t} + c_2 e^{4 \pi i a t} + c_3
  = c_2 (e^{2\pi i a t} - \rho_1)(e^{2\pi i a t} - \rho_2),$$
with $\rho_1 \rho_2 = c_3/c_2$ and $\rho_1 + \rho_2 = -c_1/c_2$. Since
$|e^{2\pi iat}|=1$ on $\R$, it follows that for all $t_0,t_1 \in \R$
$$\frac{\prod_{n=-N}^{-1} |P(t_1+n)|}{\prod_{n=1}^N |P(t_0+n)|}
  = \prod_{j=1,2} \frac{\prod_{n=-N}^{-1} |e^{2\pi i a(t_1+n)}-\rho_j|}
                       {\prod_{n=1}^N |e^{2\pi i a(t_0+n)}-\rho_j|},$$
so it suffices to bound each factor $j=1,2$ separately. We treat the two cases
$|\rho_j|=1$ and $|\rho_j|\ne1$ in turn.

\medskip
\noindent\underline{Case $|\rho_j|=1$.} Write $\rho_j= e^{2\pi i s_j}$ with
$s_j\in[0,1)$; in this case $P$ has real zeros. Since $|e^{-2\pi i(s_j-at)}|=1$,
\begin{align*}
    \prod_{k=1}^N|e^{2 \pi i a (t + k)} - \rho_j|
    &= \prod_{k=1}^N|e^{2 \pi i a (t + k)} - e^{2 \pi i s_j}|
     = \prod_{k=1}^N |e^{2 \pi i a k} - e^{2 \pi i \{s_j - at \}}|.
\end{align*}

Fix $0<\delta<1/100$ and let $N=N_k\in\mathcal N$, with
$k\in E$ and $N_k>100$. Define
\[
\widetilde B_{\delta,N}
=
\Bigl\{u\in[0,1]:
\min\Bigl\{
\frac{\langle u\rangle}{N},
\langle u-ma\rangle,
\Bigl\langle u-\frac{m}{N}\Bigr\rangle
:1\le m\le N
\Bigr\}<\frac{\delta}{N}
\Bigr\}.
\]
This is the set where the separation condition in
Proposition~\ref{prop2.1} fails. To control both the forward
and backward products, set
\[
\begin{split}
B_{\delta,j,N}
={}&\bigl\{t\in[0,1]:
       \{s_j-at\}\in\widetilde B_{\delta,N}\bigr\}\cup
\bigl\{t\in[0,1]:
       \{at-s_j\}\in\widetilde B_{\delta,N}\bigr\}.
\end{split}
\]

Writing $|\cdot|$ for Lebesgue measure, we have
\[
\bigl|\{u\in[0,1]:\langle u-x\rangle<r\}\bigr|=2r,
\qquad x\in\R,\quad 0<r<1/2.
\]
The union bound therefore gives
\[
|\widetilde B_{\delta,N}|
\le 2\delta+\sum_{m=1}^{N}\frac{4\delta}{N}
=6\delta.
\]
Moreover, for every measurable $A\subseteq[0,1]$,
\[
\bigl|\{t\in[0,1]:
\{\sigma(s_j-at)\}\in A\}\bigr|
\le C_a|A|,
\qquad \sigma\in\{-1,1\},
\qquad C_a=\frac{\lceil a\rceil}{a}.
\]
Indeed, each map has slope of modulus $a$ on its branches
and at most $\lceil a\rceil$ preimages of almost every point.
Consequently,
\[
|B_{\delta,j,N}|
\le 2C_a|\widetilde B_{\delta,N}|
\le 12C_a\delta.
\]

For $t\in[0,1]\setminus B_{\delta,j,N}$, both
$u=\{s_j-at\}$ and $v=\{at-s_j\}$ satisfy the hypotheses of
Proposition~\ref{prop2.1} with $\alpha=a$.
Since $\rho_j=e^{2\pi i s_j}$, we have
\begin{align*}
\prod_{n=1}^{N}|e^{2\pi ia(t+n)}-\rho_j|
&=\prod_{n=1}^{N}|e^{2\pi iu}-e^{2\pi ian}|,\\
\prod_{n=-N}^{-1}|e^{2\pi ia(t+n)}-\rho_j|
&=\prod_{n=1}^{N}|e^{2\pi iv}-e^{2\pi ian}|.
\end{align*}
Applying the proposition to $u$ and $v$ yields constants
$p_1,p_2>0$, independent of $N$ and $t$, such that
\begin{equation}\label{eq:circle-bounds}
p_2\le
\prod_{n=1}^{N}|e^{2\pi ia(t\pm n)}-\rho_j|
\le p_1
\end{equation}
for both choices of sign.

\medskip
\noindent\underline{Case $|\rho_j|\ne1$.}
The function
\[
\phi_j(x):=\log|e^{2\pi ix}-\rho_j|
\]
is continuously differentiable and $1$-periodic on $\R$. Set
\[
L_j:=\|\phi_j'\|_\infty
\le \frac{2\pi}{|1-|\rho_j||},
\qquad
I_j:=\int_0^1\phi_j(x)\,dx.
\]
For $N=N_k\in\mathcal N$, let $p=p_k$. Then
\[
\left|a-\frac pN\right|
\le \frac1{NN_{k+1}}
\le \frac1{N^2},
\qquad \gcd(p,N)=1.
\]

Fix $\sigma\in\{-1,1\}$. By coprimality and periodicity,
\[
\sum_{n=1}^{N}\phi_j\left(at+\sigma\frac{np}{N}\right)
=
\sum_{m=0}^{N-1}\phi_j\left(at+\frac mN\right).
\]
Comparing this shifted Riemann sum with the integral over
$[at,at+1]$ gives
\begin{equation}\label{riemsum}
\left|
\sum_{n=1}^{N}\phi_j\left(at+\sigma\frac{np}{N}\right)
-NI_j
\right|
\le L_j.
\end{equation}
Moreover, the mean value theorem on $\R$ yields
\begin{equation}\label{riemansumcomponent}
\begin{split}
&\left|
\sum_{n=1}^{N}\phi_j\bigl(a(t+\sigma n)\bigr)
-
\sum_{n=1}^{N}\phi_j\left(at+\sigma\frac{np}{N}\right)
\right|\\
&\hspace{2em}\le
L_j\left|a-\frac pN\right|\sum_{n=1}^{N}n
\le L_j\frac{N+1}{2N}
\le L_j.
\end{split}
\end{equation}
Combining these estimates, we obtain
\[
\left|
\sum_{n=1}^{N}\phi_j\bigl(a(t+\sigma n)\bigr)-NI_j
\right|
\le 2L_j,
\qquad t\in[0,1],\quad \sigma\in\{-1,1\}.
\]
Applying this with $(t,\sigma)=(t_0,1)$ and $(t_1,-1)$,
subtracting, and exponentiating gives
\begin{equation}\label{eq:offcircle-ratio-2}
C_2^{(j)}:=e^{-4L_j}
\le
\frac{\displaystyle\prod_{n=-N}^{-1}
      |e^{2\pi ia(t_1+n)}-\rho_j|}
     {\displaystyle\prod_{n=1}^{N}
      |e^{2\pi ia(t_0+n)}-\rho_j|}
\le e^{4L_j}=:C_1^{(j)}.
\end{equation}
These bounds hold for every $t_0,t_1\in[0,1]$, with constants
independent of $N$. Thus roots off the unit circle impose
no restriction on the good set.

\smallskip
\noindent\underline{Combining the two cases.}
Let
\[
J_1=\{j:|\rho_j|=1\},
\quad
J_2=\{j:|\rho_j|\ne1\}, \quad J_1\cup J_2=\{1,2\}.
\]
 Choose
\[
\delta=\min\left\{\frac1{200},\frac{\epsilon}{24C_a}\right\},
\qquad
\mathcal P(N)
=[0,1]\setminus\bigcup_{j\in J_1}B_{\delta,j,N}.
\]
Since $|B_{\delta,j,N}|\le12C_a\delta$ and $|J_1|\le2$,
\[
|[0,1]\setminus\mathcal P(N)|
\le24C_a\delta\le\epsilon.
\]
Hence $|\mathcal P(N)|\ge1-\epsilon$, and
$0<\delta<1/100$.

For $t_0,t_1\in\mathcal P(N)$, write
\[
R_{j,N}(t_0,t_1)
:=
\frac{\displaystyle\prod_{n=-N}^{-1}
      |e^{2\pi ia(t_1+n)}-\rho_j|}
     {\displaystyle\prod_{n=1}^{N}
      |e^{2\pi ia(t_0+n)}-\rho_j|}.
\]
The factorization of $P$ gives
\[
\frac{\displaystyle\prod_{n=-N}^{-1}|P(t_1+n)|}
     {\displaystyle\prod_{n=1}^{N}|P(t_0+n)|}
=
\prod_{j=1}^{2}R_{j,N}(t_0,t_1).
\]
For $j\in J_1$, denote the constants in
\eqref{eq:circle-bounds} by $p_1^{(j)},p_2^{(j)}$.
That estimate gives
\[
\frac{p_2^{(j)}}{p_1^{(j)}}
\le R_{j,N}(t_0,t_1)
\le\frac{p_1^{(j)}}{p_2^{(j)}}.
\]
For $j\in J_2$, \eqref{eq:offcircle-ratio-2} gives
\[
C_2^{(j)}
\le R_{j,N}(t_0,t_1)
\le C_1^{(j)}.
\]
Therefore, setting
\[
C_2=
\prod_{j\in J_1}\frac{p_2^{(j)}}{p_1^{(j)}}
\prod_{j\in J_2}C_2^{(j)},
\qquad
C_1=
\prod_{j\in J_1}\frac{p_1^{(j)}}{p_2^{(j)}}
\prod_{j\in J_2}C_1^{(j)},
\]
we obtain
\[
C_2
\le
\frac{\displaystyle\prod_{n=-N}^{-1}|P(t_1+n)|}
     {\displaystyle\prod_{n=1}^{N}|P(t_0+n)|}
\le C_1.
\]
The constants are positive and depend only on
$\epsilon,c_1,c_2,c_3,a$, independently of $N,t_0,t_1$.
This completes the proof.

\end{proof}

\begin{rem}\label{sequence}
    Since $C_1,C_2$ do not depend on $N$, it suffices to contradict the decay of Equation~\eqref{eq:L2decay} in a subsequence in $\mathcal N$, as \eqref{eq:L2decay} holds along every subsequence.
\end{rem}

\section{The $(3,2)$ configuration case}\label{sec:3-2config}

In this section, we focus on the proof of our main results for the  $(3,2)$ configurations. In particular, we will establish the following result:

\begin{thm}\label{thm:32}
    Let $0\neq g \in L^2(\R)$. Suppose $\Lambda_0 = \{(0,0),(0,1),(0,-1),(a,b),(a,-b)\}$ with $a b \neq 0$. Then $\mathcal{G}(g, \Lambda)$ is linearly independent. 
\end{thm}

Throughout this section $g,\Lambda,P,Q, \mathcal P, \mathcal N$ and $S$ are as constructed in
Section~\ref{sec:2}, and we set $v:=ab$. We prove Theorem~\ref{thm:32} by
treating the three regimes left open from \cite{okoudjou2019extension}:
(1) $a\notin\Q$, $b\in\Q$; (2) $a, b\notin\Q$, $v\in\Q$; and (3) $a,b,v\notin\Q$.
Together with the cases settled in \cite{okoudjou2019extension}, these establish
the $(3,2)$ configuration for every $v\neq0$ and every
$g\in L^2(\R)$. Notice that taking
moduli on both sides of Equation~\eqref{eq:ch3_relation} yields
\begin{equation}\label{eq:ch4_ab_modulus}
    |P(t)g(t)| = |Q(t)g(t-1)|, \qquad t\in S +\Z.
\end{equation}

\subsection{The case $a\notin\Q$, $b\in\Q$}\label{subsec3.1}

\begin{proof}
Suppose $b=p/q\in\Q$, with $p\in\Z$ and $q\in\N$.
Since $v=ab=ap/q$, the normalized configuration satisfies
\[
\Lambda=\{(0,0),(0,a),(0,-a),(1,v),(1,-v)\}
\subset \Z\times\frac{a}{q}\Z
=
\begin{pmatrix}
1&0\\
0&a/q
\end{pmatrix}\Z^2.
\]
This lattice has full rank because $a>0$. Consequently,
Linnell's theorem~\cite{linnell1999neumann} implies that
$\cG(g,\Lambda)$ is linearly independent for every nonzero
$g\in L^2(\R)$.
\end{proof}

\subsection{The case $a, b\notin\Q$, $v\in\Q$}\label{subsec:3.2}
\begin{proof}

Suppose that $a\notin\Q$ and $v=\frac pq\in\Q$ with $p,q$ coprime. Then $|Q|$ is
$q$-periodic on $S+\Z$,
\begin{equation}\label{eq:q-periodic}
    |Q(t+ j \pm q)| =|Q(t+j)|, \quad \forall j \in \Z,
\end{equation}
so $T(t):=\prod_{j=0}^{q-1}|Q(t+j)|$ is continuous, positive, bounded and
$q$-periodic on $S+\Z$ (which contains no zero of $Q$).

Fix $0<\epsilon<|S|/2$ and enumerate
$
\mathcal N=\{N_\ell:\ell\geq1\}.$ Let $\mathcal P(N_\ell)$ be the sets given by Lemma~\ref{Sec3}. Since
$|\mathcal P(N_\ell)|\geq1-\epsilon$, we have
$$ |S\cap\mathcal P(N_\ell)|
\geq |S|-\epsilon>0$$
for every $\ell \geq 1$. 

By reverse Fatou's lemma,
$|\limsup_{\ell\to\infty}(S\cap\mathcal P(N_\ell))|
\ge |S|-\epsilon>0$.
Choosing $t_0$ in this limsup and passing to a subsequence
with a fixed residue class modulo $q$, we obtain an infinite
set $L\subset\N$ and $r\in\{0,\ldots,q-1\}$ such that
$t_0\in S\cap\mathcal P(N_\ell)$ and $N_\ell=k_\ell q+r$
for every $\ell\in L$, with $k_\ell\to\infty$ along $L$.

For every $\ell\in L$, iterating \eqref{eq:ch4_ab_modulus} gives
\begin{align}
|g(t_0+N_\ell)|
&=
|g(t_0-1)|
\frac{\prod_{j=0}^{N_\ell}|Q(t_0+j)|}
     {\prod_{j=0}^{N_\ell}|P(t_0+j)|},
\label{eq:4.1.2forward}\\
|g(t_0-N_\ell-1)|
&=
|g(t_0-1)|
\frac{\prod_{j=-N_\ell}^{-1}|P(t_0+j)|}
     {\prod_{j=-N_\ell}^{-1}|Q(t_0+j)|}.
\label{eq:4.1.2backward}
\end{align}
Multiplying these identities and separating the factor $|P(t_0)|$,
we obtain
\begin{equation}\label{eq:4.1.2backfor}
\begin{split}
&|g(t_0+N_\ell)|\,|g(t_0-N_\ell-1)|\\
&\quad=
\frac{|g(t_0-1)|^2}{|P(t_0)|}
\frac{\prod_{j=-N_\ell}^{-1}|P(t_0+j)|}
     {\prod_{j=1}^{N_\ell}|P(t_0+j)|}
\frac{\prod_{j=0}^{N_\ell}|Q(t_0+j)|}
     {\prod_{j=-N_\ell}^{-1}|Q(t_0+j)|}.
\end{split}
\end{equation}
Since $N_\ell=k_\ell q+r$ with fixed $r$, the complete
$q$-blocks cancel by the $q$-periodicity of $Q$, yielding
\[
\frac{\prod_{j=0}^{N_\ell}|Q(t_0+j)|}
     {\prod_{j=-N_\ell}^{-1}|Q(t_0+j)|}
=
\frac{\prod_{j=0}^{r}|Q(t_0+j)|}
     {\prod_{j=-r}^{-1}|Q(t_0+j)|},
\]
where the denominator is interpreted as $1$ when $r=0$.
This ratio is positive and independent of $\ell$.
Moreover, since $t_0\in\mathcal P(N_\ell)$ for every $\ell\in L$,
Lemma~\ref{Sec3} bounds the $P$-ratio in
\eqref{eq:4.1.2backfor} below by $C_2>0$. Consequently,
\[
|g(t_0+N_\ell)|\,|g(t_0-N_\ell-1)|
\ge
\frac{C_2|g(t_0-1)|^2}{|P(t_0)|}
\frac{\prod_{j=0}^{r}|Q(t_0+j)|}
     {\prod_{j=-r}^{-1}|Q(t_0+j)|}
=:c_*>0,
\qquad \ell\in L.
\]
Here positivity follows from the nonvanishing of $g$, $P$,
and $Q$ on $S+\Z$. Since $t_0\in S$ and $N_\ell\to\infty$
along $L$, the left-hand side tends to zero by
\eqref{eq:L2decay}, contradicting this fixed positive lower
bound. Hence $\cG(g,\Lambda)$ is linearly independent.

\end{proof}

\subsection{The case $a,b,v\notin\Q$}\label{subsec:3.3}
\begin{proof}
Suppose that $a,b,v\notin\Q$, recalling that $v=ab>0$.
Write
\[
Q(t)=e^{2\pi i(-vt+\theta)}
\bigl(r_1+r_2e^{2\pi i(2vt+\theta')}\bigr),
\]
where $r_1,r_2>0$ and $\theta,\theta'\in[0,1)$.
We use the \emph{conjugate trick} from
\cite{demeter2010linear}.

\medskip
\noindent\underline{Step 1: Choosing the related points.}
For $n=0,1,\ldots$, define
\[
F_n(t)=\left\{-t-\frac{\theta'}{v}+\frac{n}{v}\right\}.
\]
Each $F_n$ preserves Lebesgue measure. Since $1/v$ is
irrational, the mean ergodic theorem for rotation by $1/v$
gives
\[
\frac1H\sum_{n=0}^{H-1}|S\cap F_n^{-1}(S)|
\longrightarrow |S|^2>0.
\]
Consequently, there are arbitrarily large $n$ for which
$|S\cap F_n^{-1}(S)|>0$. Fix such an $n'$ with
$n'>\theta'+v$, and set
\[
F:=F_{n'},
\qquad A:=S\cap F^{-1}(S).
\]

Choose $0<\epsilon<|A|/4$ and apply Lemma~\ref{Sec3}.
For each $N\in\mathcal N$, set
\[
G_N:=\mathcal P(N)\cap F^{-1}(\mathcal P(N)).
\]
Measure preservation gives $|G_N|\ge1-2\epsilon$, hence
\[
|A\cap G_N|\ge |A|-2\epsilon>0.
\]
By reverse Fatou's lemma, there exist $t_0\in A$ and an
infinite set $\mathcal N_{t_0}\subset\mathcal N$ such that
$t_0\in G_N$ for every $N\in\mathcal N_{t_0}$.
Set
\[
t_1:=F(t_0)\in S,
\qquad
m:=-t_0-\frac{\theta'}{v}+\frac{n'}{v}-t_1
=\left\lfloor-t_0-\frac{\theta'}{v}+\frac{n'}{v}\right\rfloor.
\]
Our choice of $n'$ ensures that $m\ge0$.
The points $t_0,t_1$ and the integer $m$ are now fixed.
Discarding finitely many denominators, we may assume
$N>m$ throughout $\mathcal N_{t_0}$.
Since $t_0,t_1\in\mathcal P(N)$, Lemma~\ref{Sec3} gives
\begin{equation}\label{eq:lem-ratio}
\frac{\prod_{j=-N}^{-1}|P(t_1+j)|}
     {\prod_{j=1}^{N}|P(t_0+j)|}
\ge C_2,
\qquad N\in\mathcal N_{t_0}.
\end{equation}

Moreover, for every $j\in\Z$,
\[
2v(t_1+m-j)+\theta'
=
-\bigl(2v(t_0+j)+\theta'\bigr)+2n'.
\]
Since $r_1,r_2$ are real, this implies
\[
|Q(t_1+m-j)|=|Q(t_0+j)|.
\]

\medskip
\noindent\underline{Step 2: Multiplying the orbit identities.}
For $N\in\mathcal N_{t_0}$, iterating
\eqref{eq:ch4_ab_modulus} yields
\begin{align}
|g(t_0+N)|
&=
|g(t_0-1)|
\frac{\prod_{j=0}^{N}|Q(t_0+j)|}
     {\prod_{j=0}^{N}|P(t_0+j)|},
\label{eq:ch5_forward_g}\\
|g(t_1-N+m-1)|
&=
|g(t_1-1)|
\frac{\prod_{j=-N+m}^{-1}|P(t_1+j)|}
     {\prod_{j=-N+m}^{-1}|Q(t_1+j)|}.
\label{eq:ch5_backward_g}
\end{align}
The conjugate identity gives
\[
\prod_{j=0}^{N}|Q(t_0+j)|
=
\prod_{j=-N+m}^{m}|Q(t_1+j)|
=
\frac1K\prod_{j=-N+m}^{-1}|Q(t_1+j)|,
\]
where
\[
K:=\left(\prod_{j=0}^{m}|Q(t_1+j)|\right)^{-1}
\in(0,\infty)
\]
is independent of $N$. Multiplying the orbit identities
therefore gives
\begin{equation}\label{thm2final}
\begin{split}
&|g(t_0+N)|\,|g(t_1-N+m-1)|\\
&\quad=
\frac{|g(t_0-1)|\,|g(t_1-1)|}{K}
\frac{\prod_{j=-N+m}^{-1}|P(t_1+j)|}
     {\prod_{j=0}^{N}|P(t_0+j)|}.
\end{split}
\end{equation}

The shortened backward product omits exactly the $m$
factors indexed by $-N,\ldots,-N+m-1$. Thus
\[
\prod_{j=-N+m}^{-1}|P(t_1+j)|
=
\frac{\prod_{j=-N}^{-1}|P(t_1+j)|}
     {\prod_{j=-N}^{-N+m-1}|P(t_1+j)|}
\ge
\frac{1}{\|P\|_\infty^m}
\prod_{j=-N}^{-1}|P(t_1+j)|,
\]
with the omitted product interpreted as $1$ when $m=0$.
Separating the factor $|P(t_0)|$ in the forward product
and using \eqref{eq:lem-ratio}, we obtain
\[
\frac{\prod_{j=-N+m}^{-1}|P(t_1+j)|}
     {\prod_{j=0}^{N}|P(t_0+j)|}
\ge
\frac{C_2}{\|P\|_\infty^m\,|P(t_0)|}.
\]
Consequently, \eqref{thm2final} implies
\[
|g(t_0+N)|\,|g(t_1-N+m-1)|
\ge
\frac{C_2|g(t_0-1)|\,|g(t_1-1)|}
     {K\|P\|_\infty^m\,|P(t_0)|}
=:c_*>0,
\qquad N\in\mathcal N_{t_0}.
\]
The constant is positive because $g$, $P$, and $Q$ do not
vanish on $S+\Z$. Since $t_0,t_1\in S$ and $m$ is fixed,
the left-hand side tends to zero by \eqref{eq:L2decay}
as $N\to\infty$ through $\mathcal N_{t_0}$.
This contradiction proves that $\cG(g,\Lambda)$ is
linearly independent.

\end{proof}

Theorem~\ref{thm:32} has a consequence for four-point
configurations generated by real-valued functions. To clarify
its scope, call two configurations affinely symplectically
equivalent if
\[
\Lambda_2=A\Lambda_1+z,
\qquad \det A=1,\quad z\in\R^2.
\]
After relabeling, every noncollinear four-point configuration
can be transformed in this way into
\[
\tilde\Lambda=\{(0,0),(0,1),(s,0),(a,b)\},
\qquad s\ne0.
\]
The fully collinear case is already known to yield linear
independence for every nonzero window in $L^2(\R)$.

Translation of the configuration preserves linear independence
with the window unchanged, while the linear symplectic part
is implemented by a metaplectic unitary operator $U_A$;
see \cite[Section~9.4]{Groc2001}. Consequently,
\[
\mathcal G(g,\Lambda)\text{ is linearly independent}
\quad\Longleftrightarrow\quad
\mathcal G(U_Ag,\tilde\Lambda)\text{ is linearly independent}.
\]
However, $U_Ag$ need not be real-valued when $g$ is.
Corollary~\ref{review:real} establishes the linear independence
of $\mathcal G(g,\tilde\Lambda)$ for every nonzero real-valued
$g\in L^2(\R)$ and every displayed configuration with
$a,b,s\ne0$. Thus it does not, by itself, settle the
four-point problem for arbitrary configurations and
real-valued windows.

The proof uses the restriction principle of
\cite[Corollary~4]{okoudjou2019extension}.
In a hypothetical dependence relation, the coefficient of
$g(\cdot-s)$ is nonzero, by the known independence of three
distinct time-frequency shifts. Normalizing this coefficient
to one and subtracting the complex conjugate relation
eliminates $g(\cdot-s)$ and produces a nontrivial dependence
supported on
\[
\{(0,0),(0,1),(0,-1),(a,b),(a,-b)\}.
\]
Theorem~\ref{thm:32} rules out this dependence, removing the
arithmetic restrictions on $(a,b)$ required in the earlier
application of the restriction principle.

The proof is included for completeness.

\begin{proof}[\textbf{Proof of Corollary~\ref{review:real}}]
    Suppose by contradiction that there exist nonzero coefficients $c_1,c_2,c_3$ and a nonzero real-valued $g \in L^2(\R)$ such that
    $$T_s g = c_1 g + c_2M_1 g +c_3 M_bT_a g.$$
    It follows that 
    $$T_s g = \overline{c_1} g + \overline{c_2}M_{-1} g +\overline{c_3} M_{-b}T_a g$$ where the frequency parameter $b$ was transformed to $-b$ through the reflection across the $x$-axis. 
    Hence 
    $$(c_1 -\overline{c_1})g + c_2M_1 g -\overline{c_2}M_{-1}g +c_3 M_bT_a g -\overline{c_3} M_{-b}T_a g=0.$$
    Consequently, $\cG(g,\Lambda)$ is linearly dependent with $\Lambda = \{ (0,0),(0,1),(0,-1),(a,b),(a,-b)\}$, which contradicts Theorem~\ref{thm:32}.
\end{proof}

\section{Symmetric configurations with commensurable spacings}
\label{sec:2n1}

The argument for Theorem~\ref{thm:32} extends to finitely many
commensurable frequencies because the associated Laurent polynomial
factors into finitely many linear factors. The estimates from
Lemma~\ref{Sec3} apply to each factor; only the constants and the
exceptional set depend on the degree. The polynomial $Q$
still has two terms, so the periodicity and conjugate-trick arguments
remain available. For incommensurably spaced collinear points, the
associated trigonometric polynomial need not admit this one-variable
representation.

For $n\ge1$, let
\begin{equation}\label{eq:lambdan}
\Lambda_n
=
\{(0,k):-n\le k\le n\}
\cup\{(a,b),(a,-b)\},
\qquad ab\ne0.
\end{equation}
Reflecting the time variable and relabeling the symmetric pairs,
if necessary, we may assume $a,b>0$. The unitary dilation
$D_ag(t)=a^{1/2}g(at)$ transforms $\Lambda_n$ into
\[
\widehat\Lambda_n
=
\begin{pmatrix}a^{-1}&0\\0&a\end{pmatrix}\Lambda_n
=
\{(0,ka):-n\le k\le n\}
\cup\{(1,v),(1,-v)\},
\qquad v=ab.
\]
We again denote the dilated window by $g$. A hypothetical
nontrivial dependence then takes the form
\begin{equation}\label{eq:ch6_relation}
\begin{cases}
g(t)P_n(t)=g(t-1)Q(t),\\
P_n(t)=\sum_{k=-n}^{n}c_k e^{2\pi ikat},\\
Q(t)=d_1e^{2\pi ivt}+d_2e^{-2\pi ivt},
\end{cases}
\end{equation}
for almost every $t\in\R$, where the coefficients
$c_{-n},\ldots,c_n,d_1,d_2$ are not all zero.

\subsection{Product estimates for Laurent polynomials}

\begin{lemma}[Extension of Lemma~\ref{Sec3}]\label{Sec6}
Let $a>0$ be irrational, let $n\ge1$, and suppose
\[
P_n(t)=\sum_{k=-n}^{n}c_k e^{2\pi ikat}\not\equiv0.
\]
For every $0<\epsilon<1/2$, there exist constants $C_1,C_2>0$,
depending only on $\epsilon$, $a$, and the coefficients of $P_n$,
such that for each $N\in\mathcal N(a)$ there is a measurable set
$\mathcal P_n(N)\subset[0,1]$ with
$|\mathcal P_n(N)|\ge1-\epsilon$ and
\[
C_2
\le
\frac{\displaystyle\prod_{m=-N}^{-1}|P_n(t_1+m)|}
     {\displaystyle\prod_{m=1}^{N}|P_n(t_0+m)|}
\le C_1,
\qquad t_0,t_1\in\mathcal P_n(N).
\]
\end{lemma}

\begin{proof}Set
$
k_-=\min\{k:c_k\ne0\},
\quad
k_+=\max\{k:c_k\ne0\},
\quad
d=k_+-k_-\le2n.$ 

If $d=0$, then $|P_n|$ is a positive constant, so the quotient
equals $1$ and we may take $\mathcal P_n(N)=[0,1]$.

Suppose $d>0$. With $\omega=e^{2\pi iat}$, factor
\[
P_n(t)
=
c_{k_+}\omega^{k_-}\prod_{j=1}^{d}(\omega-\rho_j).
\]
The polynomial obtained after removing $\omega^{k_-}$ has a
nonzero constant term, so every $\rho_j\ne0$. Since $|\omega|=1$,
\begin{equation}\label{eq:factored}
|P_n(t)|
=
|c_{k_+}|
\prod_{j=1}^{d}|e^{2\pi iat}-\rho_j|.
\end{equation}
Put
\[
J_1=\{j:|\rho_j|=1\},
\qquad
J_2=\{1,\ldots,d\}\setminus J_1,
\qquad
C_a=\frac{\lceil a\rceil}{a},
\]
and choose
\[
\delta
=
\min\left\{\frac1{200},\frac{\epsilon}{12dC_a}\right\}.
\]

For $j\in J_1$, the construction in the proof of
Lemma~\ref{Sec3} gives a bad set $B_{\delta,j,N}$ satisfying
\[
|B_{\delta,j,N}|\le12C_a\delta,
\]
outside which both products
\[
\prod_{m=1}^{N}|e^{2\pi ia(t\pm m)}-\rho_j|
\]
have positive upper and lower bounds independent of $N$ and $t$.
For $j\in J_2$, set
\[
\phi_j(x)=\log|e^{2\pi ix}-\rho_j|,
\qquad
L_j=\|\phi_j'\|_\infty.
\]
The estimate \eqref{eq:offcircle-ratio-2} gives
\[
e^{-4L_j}
\le
\frac{\displaystyle\prod_{m=-N}^{-1}
      |e^{2\pi ia(t_1+m)}-\rho_j|}
     {\displaystyle\prod_{m=1}^{N}
      |e^{2\pi ia(t_0+m)}-\rho_j|}
\le e^{4L_j}
\]
for every $t_0,t_1\in[0,1]$.

Define
\[
\mathcal P_n(N)
=
[0,1]\setminus\bigcup_{j\in J_1}B_{\delta,j,N}.
\]
Then
\[
|[0,1]\setminus\mathcal P_n(N)|
\le12dC_a\delta\le\epsilon.
\]
For $t_0,t_1\in\mathcal P_n(N)$, each root contributes a
positive upper and lower bound for its product quotient.
Multiplying these finitely many bounds and using
\eqref{eq:factored}, with the powers of $|c_{k_+}|$ cancelling,
proves the result.
\end{proof}

\begin{rem}\label{rem:notuniform}
The constants may depend on the degree $d$; no uniformity in
$n$ is asserted. Individual coefficients may vanish, provided
$P_n\not\equiv0$.
\end{rem}

\subsection{Vanishing coefficients}\label{subsec:degen}

The product estimate allows arbitrary zero coefficients in $P_n$.
The following lemma supplies the remaining nondegeneracy conditions.

\begin{lemma}\label{lem:vanish}
Suppose $0\ne g\in L^2(\R)$ satisfies
\eqref{eq:ch6_relation}, with $a,b>0$ and coefficients that are
not all zero. Then $P_n\not\equiv0$ and $d_1d_2\ne0$.
\end{lemma}

\begin{proof}
If either $P_n$ or $Q$ vanishes identically, nontriviality of
the coefficients and distinctness of the frequencies imply that
the other polynomial is nonzero. A nonzero exponential polynomial
has a zero set of measure zero, so \eqref{eq:ch6_relation} would
force $g=0$ almost everywhere. Thus $P_n,Q\not\equiv0$.

Suppose exactly one of $d_1,d_2$ is nonzero, corresponding to
the point $(1,\sigma v)$, where $\sigma\in\{-1,1\}$.
The remaining points belong to the full-rank lattice
\[\begin{bmatrix}0 & 1\\ a & \sigma v\end{bmatrix}\Z^2.\] 

Linnell's theorem~\cite{linnell1999neumann} rules out a nontrivial
dependence supported on this lattice. Hence both $d_1$ and $d_2$
are nonzero.
\end{proof}

\subsection{Proof of Theorem~\ref{review:two-lines}}

\begin{proof}
Suppose, toward a contradiction, that a nontrivial dependence
exists, and use the normalized notation above.
By Lemma~\ref{lem:vanish},
\[
P_n\not\equiv0,
\qquad d_1d_2\ne0.
\]
Construct a measurable set $S\subset[0,1]$ of positive measure
as in Section~\ref{sec:2}, so that the recurrence
\eqref{eq:ch6_relation} holds at every point of $S+\Z$,
the functions $g$, $P_n$, and $Q$ do not vanish there, and
the uniform decay statement \eqref{eq:L2decay} holds.
We consider four cases.

\medskip
\noindent\underline{(1) $b\in\Q$.}
Write $b=p/q$, with $q\in\N$. Since $v=ap/q$,
\[
\widehat\Lambda_n
\subset
\Z\times\frac aq\Z.
\]
This is a full-rank lattice, so Linnell's
theorem~\cite{linnell1999neumann} establishes the HRT in this case. 

\medskip
\noindent\underline{(2) $a\notin\Q$ and $v=ab\in\Q$.}
Write $v=p/q$, with $q\in\N$, so that $Q$ is $q$-periodic.
Choose $0<\epsilon<|S|/2$ and apply Lemma~\ref{Sec6}.
Using the point-and-residue subsequence selection from
Section~\ref{subsec:3.2}, choose $t_0\in S$, a fixed
$r\in\{0,\ldots,q-1\}$, and an increasing sequence
$N_\ell\in\mathcal N(a)$ such that
\[
t_0\in\mathcal P_n(N_\ell),
\qquad
N_\ell\equiv r\pmod q.
\]
The same orbit iteration and periodic cancellation as in that
section give
\[
\begin{split}
&|g(t_0+N_\ell)|\,|g(t_0-N_\ell-1)|\\
&\quad\ge
\frac{C_2|g(t_0-1)|^2}{|P_n(t_0)|}
\frac{\prod_{j=0}^{r}|Q(t_0+j)|}
     {\prod_{j=-r}^{-1}|Q(t_0+j)|}
>0.
\end{split}
\]
An empty product is understood to equal $1$.
The right-hand side is independent of $\ell$, contradicting
\eqref{eq:L2decay}.

\medskip
\noindent\underline{(3) $a,b,v\notin\Q$.}
Apply the conjugate-trick argument of
Section~\ref{subsec:3.3}, using $\mathcal P_n(N)$ in place
of $\mathcal P(N)$ and Lemma~\ref{Sec6} in place of
Lemma~\ref{Sec3}. The measure-preserving map and the
conjugate identity depend only on the two-term polynomial $Q$.
The same selection therefore gives fixed points $t_0,t_1\in S$,
an integer $m\ge0$, a constant $K>0$, and an infinite set of
denominators $N\in\mathcal N(a)$ for which
$t_0,t_1\in\mathcal P_n(N)$.
For these denominators, after discarding those with $N\le m$,
the corrected boundary-factor estimate gives
\[
|g(t_0+N)|\,|g(t_1-N+m-1)|
\ge
\frac{C_2|g(t_0-1)|\,|g(t_1-1)|}
     {K\|P_n\|_\infty^m\,|P_n(t_0)|}
>0.
\]
All quantities on the right are fixed, whereas the left-hand
side tends to zero by \eqref{eq:L2decay}. This is a contradiction.

\medskip
\noindent\underline{(4) $a\in\Q$ and $b\notin\Q$.}
Write $a=p/q$, with $q\in\N$. Then $P_n$ is $q$-periodic,
while $2v=2ab$ is positive and irrational.
Set
\[
R(t):=d_1e^{2\pi i(2v)t}+d_2.
\]
Since $Q(t)=e^{-2\pi ivt}R(t)$, we have $|Q(t)|=|R(t)|$.
Choose $0<\epsilon<|S|/2$ and apply Lemma~\ref{Sec6}
to $R$ at frequency $2v$. It gives a constant $C_Q>0$
and, for each $N\in\mathcal N(2v)$, a measurable set
$\mathcal Q(N)\subset[0,1]$ with
$|\mathcal Q(N)|\ge1-\epsilon$ such that
\begin{equation}\label{eq:Qratio}
\frac{\displaystyle\prod_{j=-N}^{-1}|Q(t+j)|}
     {\displaystyle\prod_{j=1}^{N}|Q(t+j)|}
\le C_Q,
\qquad t\in\mathcal Q(N).
\end{equation}

Using the same reverse-Fatou and residue-class selection as in
Section~\ref{subsec:3.2}, choose $t_0\in S$ and, after passing
to and relabeling a subsequence, an increasing sequence
$N_\ell\in\mathcal N(2v)$ such that
\[
t_0\in\mathcal Q(N_\ell),
\qquad
N_\ell=k_\ell q+r,
\]
where $r\in\{0,\ldots,q-1\}$ is fixed and $k_\ell\to\infty$.

Iterating the modulus of \eqref{eq:ch6_relation} gives
\begin{align}
|g(t_0+N_\ell)|
&=
|g(t_0-1)|
\frac{\prod_{j=0}^{N_\ell}|Q(t_0+j)|}
     {\prod_{j=0}^{N_\ell}|P_n(t_0+j)|},
\notag\\
|g(t_0-N_\ell-1)|
&=
|g(t_0-1)|
\frac{\prod_{j=-N_\ell}^{-1}|P_n(t_0+j)|}
     {\prod_{j=-N_\ell}^{-1}|Q(t_0+j)|}.
\label{eq:ch6_backward}
\end{align}
Multiplying these identities yields
\[
\begin{split}
&|g(t_0+N_\ell)|\,|g(t_0-N_\ell-1)|\\
&\quad=
|g(t_0-1)|^2
\frac{\prod_{j=-N_\ell}^{-1}|P_n(t_0+j)|}
     {\prod_{j=0}^{N_\ell}|P_n(t_0+j)|}
\frac{\prod_{j=0}^{N_\ell}|Q(t_0+j)|}
     {\prod_{j=-N_\ell}^{-1}|Q(t_0+j)|}.
\end{split}
\]
By periodicity,
\[
\frac{\prod_{j=-N_\ell}^{-1}|P_n(t_0+j)|}
     {\prod_{j=0}^{N_\ell}|P_n(t_0+j)|}
=
\frac{\prod_{j=-r}^{-1}|P_n(t_0+j)|}
     {\prod_{j=0}^{r}|P_n(t_0+j)|}
=:D_r(t_0)>0.
\]
Here, as before, an empty product equals $1$.
The quantity $D_r(t_0)$ is independent of $\ell$.
Moreover, \eqref{eq:Qratio} implies
\[
\begin{split}
\frac{\prod_{j=0}^{N_\ell}|Q(t_0+j)|}
     {\prod_{j=-N_\ell}^{-1}|Q(t_0+j)|}
&=
|Q(t_0)|
\frac{\prod_{j=1}^{N_\ell}|Q(t_0+j)|}
     {\prod_{j=-N_\ell}^{-1}|Q(t_0+j)|}\\
&\ge \frac{|Q(t_0)|}{C_Q}.
\end{split}
\]
Consequently,
\[
|g(t_0+N_\ell)|\,|g(t_0-N_\ell-1)|
\ge
|g(t_0-1)|^2D_r(t_0)\frac{|Q(t_0)|}{C_Q}
>0.
\]
This fixed positive lower bound contradicts
\eqref{eq:L2decay} as $\ell\to\infty$.

The four cases exhaust all $a,b>0$, so no nontrivial dependence
exists for $\widehat\Lambda_n$, and hence none exists for $\Lambda_n$.
For any finite nonempty $F\subset\Z$, choose $n\ge1$ such that
$F\subset\{-n,\ldots,n\}$. The system associated with
\[
\{(0,k):k\in F\}\cup\{(a,b),(a,-b)\}
\]
is then a subsystem of the linearly independent system associated
with $\Lambda_n$. Finally, for arbitrary distinct $b_1,b_2$, the shear
\[
S_z(x,y)=(x,y+zx),
\qquad z=-\frac{b_1+b_2}{2a},
\]
fixes the first line pointwise and sends the second pair to
$\{(a,b),(a,-b)\}$, where $b=(b_1-b_2)/2\ne0$.
Metaplectic covariance therefore gives the full statement of
Theorem~\ref{review:two-lines}.

\end{proof}

\section{Extension Principle Revisited}\label{sec:E}

Following the extension principle of \cite{okoudjou2019extension}, we
study which fourth points produce linear dependence when the window and
the first three points are fixed. Let $0\neq g\in L^2(\R)$ and let
$\Lambda_3\subset\R^2$ consist of three distinct points. For
$\lambda=(a,b)$, write $\pi(\lambda)=M_bT_a$. Since the HRT conjecture
holds for three distinct points
\cite{heil1996linear,linnell1999neumann}, the space
\[
V=\operatorname{span}\cG(g,\Lambda_3)
\]
has dimension three.

\begin{deft}\label{def:E}
The extension-failure set of $g$ and $\Lambda_3$ is
\[
\mathcal E(g,\Lambda_3)
=\{\lambda\in\R^2:\pi(\lambda)g\in V\}.
\]
It contains $\Lambda_3$, and its remaining points are precisely the
fourth points $\lambda$ for which
$\cG(g,\Lambda_3\cup\{\lambda\})$ is linearly dependent.
\end{deft}

When $\|g\|_2=1$, this is the level set where the extension function
of \cite{okoudjou2019extension} equals $1$. The four-point
counterexamples in \cite{dai2026minimum,oussa2026fourpoint} show that
$\mathcal E(g,\Lambda_3)$ can strictly contain $\Lambda_3$.

\begin{prop}\label{prop:Ebasic}
With the notation above, set $\mathcal E=\mathcal E(g,\Lambda_3)$.
Then:
\begin{enumerate}[label=(\roman*)]
    \item $\mathcal E$ is compact.
    \item Any $m\geq4$ distinct points of $\mathcal E$ generate a
    linearly dependent Gabor system with window $g$.
    \item For every $\nu\in\R^2$,
    \[
    \mathcal E(g,\Lambda_3+\nu)
    =\mathcal E(g,\Lambda_3)+\nu.
    \]
\end{enumerate}
\end{prop}

\begin{proof}
The map $\lambda\mapsto\operatorname{dist}(\pi(\lambda)g,V)$ is
continuous, so its zero set $\mathcal E$ is closed. Boundedness follows
from \cite[Corollary~2]{okoudjou2019extension}, proving (i).
Statement (ii) follows because all the corresponding vectors lie in
the three-dimensional space $V$.

For (iii), $\pi(\lambda+\nu)$ is a unimodular multiple of
$\pi(\nu)\pi(\lambda)$, and therefore
\[
\operatorname{span}\cG(g,\Lambda_3+\nu)=\pi(\nu)V.
\]
Thus $\pi(\lambda+\nu)g\in\pi(\nu)V$ if and only if
$\pi(\lambda)g\in V$, which gives the asserted identity.
\end{proof}

By Proposition~\ref{prop:Ebasic}(ii), configurations of at least four
points whose Gabor systems are linearly independent cannot occur
inside $\mathcal E$. In particular, the known collinear and $(2,2)$
cases impose the following restrictions.

\begin{prop}\label{prop:Econstraints}
The set $\mathcal E=\mathcal E(g,\Lambda_3)$ satisfies:
\begin{enumerate}[label=(C\arabic*)]
    \item Every line intersects $\mathcal E$ in at most three points.
    \item For each direction, at most one line in that direction
    contains two or more points of $\mathcal E$.
\end{enumerate}
\end{prop}

\begin{proof}
Four collinear points generate a linearly independent Gabor system
\cite{heil1996linear}, as do four points arranged in pairs on two
distinct parallel lines \cite{demeter2012proof}. Either configuration
inside $\mathcal E$ would contradict
Proposition~\ref{prop:Ebasic}(ii).
\end{proof}

We can now prove Theorem~\ref{review:extensions}.


\begin{proof}[Proof of Theorem~\ref{review:extensions}]

Write $\mathcal E=\mathcal E(g,\Lambda_3)$.

\emph{Measure zero.}
The set $\mathcal E$ is compact, hence Borel, and each vertical fibre
\[
\mathcal E_x=\{y\in\R:(x,y)\in\mathcal E\}
\]
has at most three points by (C1). Tonelli's theorem therefore gives
\[
|\mathcal E|=\int_{\R}|\mathcal E_x|\,dx=0.
\]

\emph{Hausdorff dimension.}
Write $\mathcal H^t$ for $t$-dimensional Hausdorff measure. Suppose
$\dim_{\mathrm H}\mathcal E>1$ and choose
$1<t<\dim_{\mathrm H}\mathcal E$, so that
$\mathcal H^t(\mathcal E)=\infty$. Since $\mathcal E$ is compact,
Besicovitch's theorem \cite{besicovitch1952} provides a compact
subset $K\subset\mathcal E$ with
\[
0<\mathcal H^t(K)<\infty.
\]
Marstrand's slicing theorem \cite{marstrand1954} then yields a line
$\ell$ such that
\[
\dim_{\mathrm H}(K\cap\ell)=t-1>0.
\]
This contradicts (C1), since $K\cap\ell$ has at most three points.
Hence $\dim_{\mathrm H}\mathcal E\leq1$.

\emph{Total disconnectedness.}
Suppose $\mathcal E$ contains a connected set with more than one
point. Taking its closure, we obtain a compact connected set
$C\subset\mathcal E$ with more than one point. Let $p(x,y)=x$.
Then $p(C)=[\alpha,\beta]$ for some $\alpha<\beta$: otherwise $C$
would lie in a vertical line and would be a singleton by (C1)
and connectedness.

By (C2), at most one vertical fibre of $\mathcal E$ contains more
than one point. Choose a nondegenerate closed interval
$J=[s,t]\subset(\alpha,\beta)$ avoiding its abscissa, if such a
fibre exists. The restriction
\[
p|_{C_J}:C_J\longrightarrow J,
\qquad C_J=C\cap p^{-1}(J),
\]
is a continuous bijection from a compact space onto a Hausdorff
space, hence a homeomorphism. Consequently,
\[
C_J=\{(x,f(x)):x\in J\}
\]
for a continuous function $f:J\to\R$.

Set
\[
m=\frac{f(t)-f(s)}{t-s},
\qquad h(x)=f(x)-mx.
\]
Then $h(s)=h(t)$. If $h$ were constant, the graph of $f$ would lie
in a line, contradicting (C1). Thus there is $u\in(s,t)$ with
$h(u)\neq h(s)$. Choose $c$ strictly between these two values.
The intermediate value theorem gives
$x_1\in(s,u)$ and $x_2\in(u,t)$ such that
$h(x_1)=h(x_2)=c$. The distinct parallel lines
\[
y=mx+h(s)
\qquad\text{and}\qquad
y=mx+c
\]
therefore each contain two points of $C_J\subset\mathcal E$,
contradicting (C2); see Figure~\ref{fig:chords}.
Hence $\mathcal E$ is totally disconnected.

\emph{Almost-everywhere independence.}
For $\mu=(\mu_1,\ldots,\mu_4)\in(\R^2)^4$, define
\[
G_g(\mu)
=\bigl(\langle\pi(\mu_j)g,\pi(\mu_k)g\rangle\bigr)_{j,k=1}^{4},
\qquad
\mathcal D_g=\{\mu:\det G_g(\mu)=0\}.
\]
The set $\mathcal D_g$ records linear dependence and is closed,
because $\lambda\mapsto\pi(\lambda)g$ is continuous in $L^2(\R)$.
Triples $\mu'=(\mu_1,\mu_2,\mu_3)$ with repeated entries form a
Lebesgue-null subset of $\R^6$. For every other triple, three-point
independence gives
\[
(\mathcal D_g)_{\mu'}
:=\{\mu_4:(\mu',\mu_4)\in\mathcal D_g\}
=\mathcal E(g,\{\mu_1,\mu_2,\mu_3\}),
\]
which has measure zero by the first part of the proof. Tonelli's
theorem now yields
\begin{equation}\label{eq:TonelliD}
|\mathcal D_g|
=\int_{\R^6}|(\mathcal D_g)_{\mu'}|\,d\mu'=0,
\end{equation}
proving the final assertion.
\end{proof}

\begin{figure}[t]
\centering
\begin{tikzpicture}[scale=1.1,
    point/.style={circle,fill=black,inner sep=1.8pt},
    hollow/.style={circle,draw=black,fill=white,inner sep=1.8pt}]

\draw[thick,domain=0:5,samples=80]
    plot (\x,{-0.4*\x*\x+2.2*\x+0.4});

\draw[gray] (-0.5,0.3) -- (6.6,1.72)
    node[right] {$y=mx+h(s)$};
\draw[gray,dashed] (-0.5,1.8) -- (6.6,3.22)
    node[right] {$y=mx+c$};

\draw[gray!60] (-0.5,-0.4) -- (6.6,-0.4);
\foreach \x/\y/\lab in {
    0/0.4/s,
    0.919/2.084/x_1,
    2.5/3.4/u,
    4.081/2.716/x_2,
    5/1.4/t}{
    \draw[gray!60,dotted] (\x,-0.4) -- (\x,\y);
    \node[below] at (\x,-0.4) {$\lab$};
}

\node[point] at (0,0.4) {};
\node[below left] at (0.05,0.3) {$(s,f(s))$};
\node[point] at (5,1.4) {};
\node[below right] at (5.05,1.3) {$(t,f(t))$};
\node[hollow] at (0.919,2.084) {};
\node[left] at (0.8,2.15) {$(x_1,f(x_1))$};
\node[hollow] at (4.081,2.716) {};
\node[above right] at (4.25,3.0) {$(x_2,f(x_2))$};
\node[point] at (2.5,3.4) {};
\node[above] at (2.5,3.45) {$(u,f(u))$};

\end{tikzpicture}
\caption{The endpoint chord and a distinct parallel line each
contain two points of the graph, contradicting (C2).}
\label{fig:chords}
\end{figure}
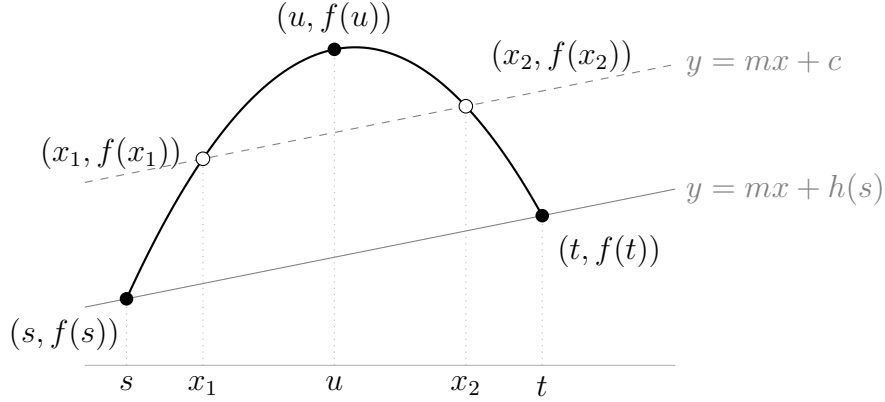

\section*{Acknowledgements}
This work was partially supported by the National Science Foundation, grants DMS-2205771 and DMS-2309652. 

The authors thank V. Oussa for his helpful comments and corrections to earlier versions of this manuscript. In particular, he pointed out that the original argument did not establish the claimed result for arbitrary four-point configurations and real-valued windows, since metaplectic transformations need not preserve real-valuedness. He also drew the second author’s attention to the fact that $E$ need not contain infinitely many odd nonnegative integers and shared a correction based on the reflection $\alpha\mapsto 1-\alpha$. The argument presented here to address this parity issue was suggested by GPT Astra Max 6. AI tools also helped identify inconsistencies in the manuscript and assisted with figure preparation, grammar, spelling, and stylistic editing. The authors checked all mathematical arguments and corrections in detail and take full responsibility for the mathematical content, proofs, and conclusions.
\bibliographystyle{amsplain}
\bibliography{reference.bib}
\end{document}